\documentclass[11pt,leqno]{article}
\usepackage{enumerate}
\usepackage{amsmath}
\usepackage{amssymb}
\usepackage[english]{babel}
\usepackage{microtype}
\usepackage{authblk}

\usepackage{fullpage}

\newtheorem{thm}{Theorem}[section]
\newtheorem{lem}[thm]{Lemma}
\newtheorem{conj}[thm]{Conjecture}
\newtheorem{cor}[thm]{Corollary}
\newtheorem{defi}[thm]{Definition}
\newtheorem{fact}[thm]{Fact}
\newtheorem{eg}[thm]{Example}

\newcounter{claim}

\newenvironment{proof}[1][]%
 {\noindent {\setcounter{claim}{0}\sc proof ---
   }{#1}{}}{\hfill$\Box$\vspace{2ex}}

\newenvironment{claim}[1][]%
{\refstepcounter{claim}\vspace{1ex}\noindent{\it{\textit{Claim }\arabic{claim}{#1}{: }}}\it}{\vspace{1ex}}

\newenvironment{proofclaim}[1][]%
	{\noindent {}{#1}{}}{ This proves Claim~\arabic{claim}.\vspace{1ex}}

\newcommand{\ov}{\overline}

\newcommand{\cH}{{\cal H}}

\newcommand{\mc}{\mathcal}
\newcommand{\sm}{\setminus}
\newcommand{\prc}{\preccurlyeq}


\title{Lines, betweenness and metric spaces}
\author[1]{Pierre Aboulker}
\author[2]{Xiaomin Chen}
\author[3]{Guangda Huzhang}
\author[4]{\\ Rohan Kapadia}
\author[4]{Cathryn Supko}
\affil[1]{Universidad Andres Bello (Santiago, Chile)\footnote{supported by Fondecyt Postdoctoral grant 3150314 of CONICYT Chile}}
\affil[2]{Shanghai Jianshi LTD (Shanghai, China)}
\affil[3]{Shanghai Jiao Tong University (Shanghai, China)}
\affil[4]{Concordia University (Montr\'eal, Canada)}

\begin{document}

\maketitle

\begin{abstract}
A classic theorem of Euclidean geometry asserts that any noncollinear set of $n$ points in the plane determines at least $n$ distinct lines. Chen and Chv\'atal conjectured that this holds for an arbitrary finite metric space, with a certain natural definition of lines in a metric space.

We prove that in any metric space with $n$ points, either there is a line containing all the points
or there are at least
$\Omega(\sqrt{n})$ lines. This is the first polynomial lower bound on the number of lines in general finite metric spaces.
In the more general setting of pseudometric betweenness, we prove a corresponding bound of $\Omega(n^{2/5})$ lines.
When the metric space is induced by a connected graph, we prove that either there is a line containing all the points
or there are $\Omega(n^{4/7})$ lines, improving the previous $\Omega(n^{2/7})$ bound.
We also prove that the number of lines in an $n$-point metric space is at least $n / 5w$,
where $w$ is the number of different distances in the space, and we give an $\Omega(n^{4/3})$
lower bound on the number of lines in metric spaces induced by graphs with constant diameter,
as well as spaces where all the positive distances are from \{1, 2, 3\}.
\end{abstract}

\section{Introduction}

A classic theorem in plane geometry states that any $n$ points in the Euclidean plane determine at least $n$ lines, unless they are all on the same line.
As noted by Erd\H{os} (\cite{E43}), this fact is a corollary of the Sylvester-Gallai theorem, which asserts that, for every non-collinear set $S$ of $n$ points in the plane, some line goes through precisely two points of $S$. It is also a special case of a well known combinatorial theorem later proved by De Bruijn and Erd\H{o}s (\cite{DBE}).

Coxeter \cite{Cox} gave a proof of the Sylvester-Gallai theorem using \textit{ordered geometry}: that is, without using notions of measurement of distances or measurement of angles, but instead employing the ternary relation of \textit{betweenness}. A point $b$ is said to \textit{lie between} points $a$ and $c$ if $b$ is an interior point of the line segment linking $a$ and $c$.
We write $[abc]$ for the statement that $b$ lies between $a$ and $c$ and we say that $\{a,b,c\}$ is \emph{collinear}.
In this notation, a \textit{line} $\ov{ab}$ is defined (for any two distinct points $a$ and $b$) as
\begin{equation} \label{def}
\ov{ab} = \{a, b\} \cup \{ c: \{a, b, c\} \text{ is collinear}\}.
\end{equation}

For three distinct points $a, b, c$ in an arbitrary metric space $(V, \rho)$, we say that $b$ lies \emph{between} $a$ and $c$ if
\[ \rho(a, c) = \rho(a, b) + \rho(b, c) \]
and we write $[abc]$ when this holds. Using this definition of betweenness, the \emph{line generated} by two points $a$ and $b$, denoted $\ov{ab}$, is defined to be the set consisting of $a$, $b$ and all other points $c$ satisfying one of the following:
\[
\begin{split}
\rho(a, b) & = \rho(a, c) + \rho(c, b) \Leftrightarrow [acb]\Leftrightarrow [bca]\\
\rho(a, c) & = \rho(a, b) + \rho(b, c)\Leftrightarrow [abc] \Leftrightarrow [cba]\\
\rho(b, c) & = \rho(b, a) + \rho(a, c)\Leftrightarrow [bac] \Leftrightarrow [cab].
\end{split}
\]
A line that contains all points of $V$ is called \emph{universal}.
With this definition of lines in a metric space, Chen and Chv\'{a}tal conjectured:

\begin{conj}\label{conj.cc} (\cite{CC}) In any finite metric space $(V, \rho)$ with $|V| \geq 2$, either there is a universal line or there are at least $|V|$ distinct lines.
\end{conj}

In this paper, we make progress towards this conjecture by proving the first known polynomial lower bound on the number of lines in a finite metric space. We prove that (see Theorem \ref{thm.main})
any finite metric space on $n$ points ($n \geq 2$) with no universal line has at least $\left(\frac{1}{\sqrt{2}} - o(1)\right)\sqrt{n}$ distinct lines.

We also show that we can improve this bound for some particular types of metric spaces: those with a constant number of distinct distances, those induced by graphs and those induced by graphs of constant diameter.
We now detail these new bounds, placing them in the context of existing results.

Chiniforooshan and Chv\' atal \cite{Chini_Chvatal} proved that an $n$-point metric space where all distances are in the set $\{0, 1, 2\}$ has $\Omega(n^{4/3})$ lines.
This was also the first type of metric space for which Conjecture~\ref{conj.cc} was shown to be true:

\begin{thm} \label{thm.12metric} (Chv\'{a}tal \cite{Chvatal})
Any metric space on $n$-points, for $n \geq 2$, in which all positive distances are either one or two, either has at least $n$ lines or has a universal line.
\end{thm}

Extending the result of  Chiniforooshan and Chv\' atal, we prove that an $n$-point metric space where all distances are in  $\{0, 1, 2, 3\}$ has $\Omega(n^{4/3})$ lines (see Theorem~\ref{thm.3metric}).
More generally, we prove that any $n$-point metric space with a constant number of distinct distances has $\Omega(n)$ lines (see Theorem~\ref{thm.bounded_dists}) and we conjecture the following:

\begin{conj}\label{const.dist}
Any $n$-point metric space $(n \geq 2)$ with a constant number of distinct distances has $\Omega(n^{4/3})$ lines.
\end{conj}

Metric spaces can be viewed as weighted connected graphs in which the edges are weighted by positive real numbers and the distance between any two vertices is defined to be the smallest total weight of a path between them. A special case is when all the edges are unit weighted, in which case we have \emph{metric spaces induced by connected graphs}, or simply called {\em graph metrics}.

We prove that, unless there is a universal line, the number of lines in an $n$-point graph metric is $\Omega(n^{4/7})$ (Theorem \ref{thm.graphs}), improving the $\Omega(n^{2/7})$ bound proved in \cite{Chini_Chvatal}.
We also prove that metric spaces induced by constant diameter graphs have $\Omega(n^{4/3})$ lines (see Theorem \ref{thm.d_graphs}), verifying Conjecture~\ref{const.dist} in the case of graph metrics.

\medskip

Conjecture~\ref{conj.cc} has also been proved for the following particular classes of metric spaces:
\begin{itemize}
\item metric spaces consisting of $n$ points in general position in the plane with the $L_1$ metric~(\cite{KP}),
\item metric spaces induced by chordal graphs~(\cite{BBCCCCFZ}),
\item  metric spaces induced by distance-hereditary graphs~(\cite{AK}), and
\item  graph metrics where no line is a proper subset of another~(\cite{CHMY}).
\end{itemize}

Conjecture~\ref{conj.cc} may be true in a more general setting than metric spaces, as we now explain.
Let $\mc B$ be a ternary relation on a set $V$ and write $[abc]$ to mean that $abc \in \mc B$.
The relation $\mc B$ is called a \emph{metric betweenness} if there is a metric $\rho$ on $V$ such that $[abc]$ if and only if
\[
a, b, c \text{ are distinct  and } \rho(a,b) + \rho(b,c) = \rho(a,c).
\]
Menger~\cite{menger} seems to have been the first to study  metric betweenness. He proved that, in addition to the obvious properties:
\begin{itemize}
\item[(M0)]  if $[abc]$, then $a$, $b$, $c$ are distinct,
\item[(M1)] if $[abc]$, then $[cba]$,
\item[(M2)] if $[abc]$, then $[bac]$ does not hold,
\end{itemize}
every metric betweenness must satisfy the following property, called \emph{inner transitivity}:
\begin{itemize}
\item[(M3)] if $[abc]$ and $[acd]$, then $[abd]$ and $[bcd]$.
\end{itemize}

A ternary relation $\mc B$ on a set $V$ is a \emph{pseudometric betweenness} if it satisfies (M0), (M1), (M2) and (M3).
We refer to the elements of $V$ as \emph{points of the pseudometric betweenness}.
Since pseudometric betweennesses are not all metric (see~\cite{Chvatal2} for more on this subject), they are a proper generalization of metric betweennesses.

We can define the lines of a pseudometric betweenness in the same way as for metric spaces.
If $\mc B$ is a pseudometric betweenness on a set $V$, a set of three points $\{a, b, c\}$ in $V$ is called \emph{collinear} if one of $[abc], [bca]$ or $[cab]$ holds.
For distinct points $a, b \in V$, the \emph{line of $\mc B$ generated by $a$ and $b$} is defined to be
\[
 \ov{ab} = \{a, b\} \cup \{ c : [abc], [acb], \text{ or } [cab] \}.
\]
In~\cite{BBC1}, it is asked whether or not Conjecture~\ref{conj.cc} holds if we replace `metric space' with `pseudometric betweenness'.
In this paper, we prove that an $n$-point pseudometric betweenness either has a universal line or has at least $(2^{-1/5}-o(1)) n^{2/5}$ lines (see Theorem \ref{thm.pmb}).

In \cite{CC}, lines in an even more general setting were considered --- given a 3-uniform hypergraph $\cH$, and two vertices $a$ and $b$,
the line $\ov{ab}$ is defined to be the union of $\{a, b\}$ and all vertices $c$ such that $\{a, b, c\}$ is a hyperedge in $\cH$.
It is clear that the definitions conform for any metric space $(V, \rho)$ and the $3$-uniform hypergraph on $V$ whose hyperedges are the collinear triples of $(V, \rho)$ (actually, it is easy to observe that studying lines in $3$-uniform hypergraphs is equivalent to studying lines induced by ternary relations satisfying (M0)).
In \cite{CC} it was proved that for any $3$-uniform hypergraph on $n$ vertices without a universal line, there are at least $\log_2 n$ lines, and examples were given with as few as $\exp(O(\sqrt{\ln n}))$ lines (which implies that  Conjecture~\ref{conj.cc} is false if one replaces `metric space' by `$3$-uniform hypergraph' in its statement). The lower bound was recently improved to $(2 - o(1)) \log_2 n$ in \cite{ABCCCM}.
Even for finite metric spaces, which are less general than $3$-uniform hypergraphs, no better lower bound than this on the number of lines has previously been published.
Since $\exp(O(\sqrt{\ln n}))$ is asymptotically smaller than any polynomial in $n$,  our results shows that, in terms of lines as we have defined them, there are intrinsic differences between pseudometric betweennesses and arbitrary 3-uniform hypergraphs.

\medskip

We present our results from the most general setting (pseudometric betweenness) to the least general (graph metrics).
This also happens to go from the simpler proofs to the more complicated ones.
In Section~\ref{sect.pmb} we prove our lower bounds for the number of lines in a pseudometric betweenness.
In Section~\ref{sect.general} we prove our lower bounds for metric spaces.
In Section \ref{sect.few_weights}, we prove a lower bound on the number of lines in a metric space with a constant number of distinct distances.
In Section~\ref{sect.123}, we study metric spaces where all distances are in $\{0,1,2,3\}$.
Finally, in Sections~\ref{sect.pairs} and~\ref{sect.graph}, we study lines in graph metrics.

\section{The lower bound for pseudometric betweennesses}\label{sect.pmb}

Let $\mc B$ be a pseudometric betweenness on a finite set $V$.
Let $p_1, \dots, p_k$ be $k$ distinct points of $V$.
We write $[p_1p_2 \dots p_k]$ to mean that $[p_rp_sp_t]$ holds for any $1 \le r < s < t \le k$, and in this case we say that the sequence $(p_1,p_2, \dots, p_k)$ is a \emph{geodesic}.
Moreover, we call a set of points $\{p_1, p_2, ..., p_k\}$ geodesic if some permutation of it forms a geodesic sequence. We note that, on the graph metrics, our notion conforms to the usual notion of {\em graph geodesic} (\cite{H}).

We will make frequent use of the following facts, whose easy proofs we omit.

\begin{fact}\label{fact.pseudometric} Let $\mc B$ be a pseudometric betweenness on a set $V$. For distinct points $a$, $b$, $c$, $d$, and $p_1, \dots, p_k$ ($k\ge 3$),
\begin{enumerate}[(a)]

\item if $[abc]$ and $[acd]$, then $[abcd]$;

\item if $[abd]$ and $[bcd]$, then $[abcd]$;

\item if $[p_1p_2...p_k]$, then $[p_{i_1} p_{i_2} ... p_{i_t}]$ for any $1 \leq i_1 < i_2 <... < i_t \leq k$.

\end{enumerate}
Moreover, if $\mc B$ is a metric  betweenness and $(V,\rho)$ is an associated  metric space, then
\begin{enumerate}
 \item[(d)] if $[p_1p_2...p_k]$, then $\rho(p_s, p_t) = \sum_{i=s}^{t-1} \rho(p_i, p_{i+1})$ for any $1 \le s < t \le k$.
\end{enumerate}

\end{fact}

Note that (a) and (b) follow from (M1) and (M3).
We now prove our lower bound on the number of lines in any pseudometric betweenness. We start with a technical lemma.

\begin{lem}\label{lem.long_path}
Let $\mc B$ be a pseudometric betweenness on a finite set $V$, with no universal line.
If $\mc B$ admits a geodesic sequence of length $k$, then $\mc B$ has at least $k$ distinct lines.
\end{lem}

\begin{proof}
Let $p_1, \dots, p_k$ be $k$ points of $V$ such that $[p_1\dots p_k]$.
For $i=1, \dots, k-1$, let $q_i$ be a point not in $\ov{p_{i}p_{i+1}}$ (such a point exists since $\mc B$ has no universal line).
Note that this means that  $\{q_i, p_i, p_{i+1}\}$ is not collinear.
We are going to prove that the $k$ lines $\ov{p_1q_1},\,  \ov{p_2q_2},  \dots  , \, \ov{p_{k-1}q_{k-1}}$ and $\ov{p_1p_k}$ are pairwise distinct.

Note that $\{p_1, \dots, p_k\} \subseteq \ov{p_1p_k}$. For each $i = 1, \dots, k-1$, we have $p_{i+1} \notin \ov{p_iq_i}$, so $\ov{p_iq_i} \neq \ov{p_1p_k}$.
So it suffices to prove that $\ov{p_iq_i} \neq \ov{p_jq_j}$ where $1 \le i <j \le k-1$.

If $ p_j \notin  \ov{p_iq_i}$, then $\ov{p_iq_i} \neq \ov{p_jq_j}$ and we are done.
In particular, if $j = i+1$, then $p_j \notin \ov{p_iq_i}$.
So we may assume that $j > i+1$ and $p_j \in \ov{p_iq_i}$ which means that  $\{p_j, p_i, q_i\}$ is collinear.
If $[p_jp_iq_i]$, then since $[p_j p_{i+1} p_i]$, we have $[p_j p_{i+1} p_i q_i]$ and thus $q_{i} \in \ov{p_i p_{i+1}}$, a contradiction.
Similarly, if $[p_ip_jq_i]$, then since $[p_i p_{i+1} p_j]$, we have $[p_i p_{i+1} p_j q_i]$ and so $q_i \in \ov{p_i p_{i+1}}$, a contradiction.
Hence we may assume $[p_iq_ip_j]$.
But then since $[p_i p_j p_{j+1}]$, we have $[p_iq_ip_jp_{j+1}]$  and thus $p_{j+1} \in \ov{p_iq_i}$. Therefore, the fact that $p_{j+1} \notin \ov{p_jq_j}$ implies that $\ov{p_iq_i} \neq \ov{p_jq_j}$.
\end{proof}

\begin{thm}\label{thm.pmb}
If $\mc B$ is a pseudometric betweenness on an $n$-point set $V$ ($n \geq 2$), then either $\mc B$ has a universal line or $\mc B$ has at least $(2^{-1/5}-o(1))n^{2/5}$ lines.
\end{thm}

\begin{proof}
Let $a$ be a point of $V$.
Define a binary relation $\prc$  on $V\sm \{a\}$ as follows: for any $x, y \in V\sm \{a\}$, $x \prc y$ if $[axy]$ or $x = y$. Then $\mc P = (V \sm \{a\}, \prc)$ is a poset; this follows easily from inner transitivity.
The well-known Dilworth's theorem \cite{Dilworth} implies that $\mc P$ either has a chain of size at least $(2^{-1/5} - o(1))n^{2/5}$ or an antichain of size at least $2^{1/5}n^{3/5}$.

If $\mc P$ has a chain $x_1 \prc x_2 \prc \dots \prc x_t$ where $t\ge (2^{-1/5} - o(1))n^{2/5}$, then we have $[x_1x_2 \dots x_t]$ by inner transitivity, so we are done by Lemma~\ref{lem.long_path}.
So we may assume that $\mc P$ admits an antichain $Y=\{y_1, \dots, y_k\}$ where $k \geq 2^{1/5} n^{3/5}$.
Consider the set of lines $\{ \ov{ay} : y \in Y \}$. If it has more than $2^{-1/5}n^{2/5}$ distinct elements, then we are done.
Otherwise there is a set $A \subseteq Y$ such that $|A| \geq 2^{2/5} n^{1/5}$ and such that $\ov{ax}$ is the same line for every $x \in A$.

We claim that $A$ contains no collinear triples. Let $u,v,w \in A$ and assume that $\{u, v, w\}$ is collinear. Without loss of generality, we may assume that $[uvw]$.
Since $\ov{au} = \ov{av}$, we have $v \in \ov{au}$ so $\{a, u, v\}$ is collinear. As $\{u, v\}$ is contained in the antichain $Y$, we have $[uav]$. Thus we have $[avw]$ by inner transitivity, a contradiciton to the fact that $\{v, w\}$ is contained in the antichain $Y$.
It follows that no three distinct points in $A$ are collinear, so for every $x, y \in A$ we have $\ov{xy} \cap A = \{x, y\}$.
As any two points of $A$ define a distinct line, this gives $\binom{\lceil 2^{2/5} n^{1/5} \rceil}{2} \geq (2^{-1/5}-o(1))n^{2/5}$ lines.
\end{proof}

\section{The lower bound for metric spaces}\label{sect.general}

We now prove our lower bound on the number of lines in any finite metric space.
The proof is very similar to the proof of Theorem~\ref{thm.pmb}; but by exploiting the properties of metric spaces we get a stronger bound than we got for pseudometric betweennesses.
The \emph{diameter} of a finite metric space is the maximum distance between any two of its points.

\begin{thm}\label{thm.main}
Any finite metric space on $n$ points ($n \geq 2$) with no universal line has at least $\left(\frac{1}{\sqrt{2}} - o(1)\right)\sqrt{n}$ distinct lines.
\end{thm}

\begin{proof}
Let $M = (V, \rho)$ be a finite metric space on $n$ points with no universal line.
We may assume that $n \ge 4$.
Let $D$ be the diameter of $M$ and let $a$, $b$ be two points such that $\rho(a,b) = D$.
Set:
\begin{align*}
&X_a = \{ x \in V : \rho(a, x) > D / 2 \},\\
& X_b = \{x \in V : \rho(b, x) > D / 2\}, \\
& Y = \{x \in V : \rho(a, x) = \rho(b, x) = D / 2 \}.
\end{align*}

Note that if $x \notin X_a \cup X_b$, then $\rho(a, x) = \rho(b, x) = D/2$, for otherwise we would have $\rho(a,b) \leq \rho(a, x)+\rho(b,x) < D$. So $Y = V \setminus (X_a \cup X_b)$.

We first assume that $|X_a| > (n - n^{0.9}) / 2$ and let $t = \lceil |X_a|^{1/2} \rceil$; so $t \geq (1/\sqrt{2} - o(1))\sqrt{n}$.
We define a binary relation $\preccurlyeq$ on points in $X_a$ where for any $x, y \in X_a$, $x \preccurlyeq y$ if $x = y$ or $[axy]$. It is easy to check that $\mathcal P = (X_a, \preccurlyeq)$ is a poset. By Dilworth's theorem, $\mc P$ either has a chain of size $t$, say $(x_1, x_2, ..., x_t)$, or an antichain of size $t$, say $A = \{y_1, ..., y_t\}$.
In the former case, we clearly have $[a x_1 x_2 ... x_t ]$
and thus by Lemma~\ref{lem.long_path}, $M$ has at least $t$ lines.
In the latter case, for any distinct $y_i, y_j \in A$, both $\rho(a,y_i)$ and $\rho(a,y_j)$ are greater than $D/2$ so $[y_iay_j]$ does not hold. Moreover, since $y_i$ and $y_j$ are not comparable in $\mathcal P$, neither $[ay_iy_j]$ nor $[ay_jy_i]$ holds. So $y_i \notin \ov{ay_j}$ and thus, for $i=1, \dots, t$, the lines $\ov{ay_i}$ are pairwise distinct so $M$ has at least $t$ lines.

By the same argument, we are done when $|X_b| > (n - n^{0.9}) / 2$. So now we may assume $|Y| \geq n^{0.9}$.
Consider the set of lines $\{\ov{ay} : y \in Y\}$. If it has size more than $\sqrt{n}$, then we are done. Otherwise, there is a set $A \subseteq Y$ such that $\ov{ax}$ is the same line for every $x \in A$ and such that $|A| \geq |Y| / \sqrt{n}$. For any distinct points $x$ and $y$ in $A$, the set $\{a, x, y\}$ is collinear. Then since $\rho(a, x) = \rho(a, y) = D/2$, it follows that $\rho(x, y) = D$ and $[xay]$. It follows that no three distinct points in $A$ are collinear, so for every $x, y \in A$ we have $\ov{xy} \cap A = \{x, y\}$. Therefore, the metric space has at least $\binom{|A|}{2}$ distinct lines, which is $\Omega(n^{0.8})$.
\end{proof}

\section{Metric spaces with a bounded number of distances}\label{sect.few_weights}

In this  section  we prove a linear lower bound  on the number of lines in metric spaces with a bounded number of distinct distance values (Theorem~\ref{thm.bounded_dists}).

We introduce certain graph-theoretic concepts that we will use in this proof; these definitions will also be needed when we study metric spaces induced by graphs later in this paper.
We define a graph $G$ to be a pair $(V, E)$ where $V$ is a finite {\em vertex set} and $E \subseteq \binom{V}{2}$ is the {\em edge set}.
We often denote by $uv$ the edge $\{u, v\}$. The vertex and edge sets of a graph $G$ are denoted by $V(G)$ and $E(G)$, respectively.
For a vertex $v$ in a graph, we write $N(v)$ for the set of its {\em neighbours} $\{u : uv \in E\}$ and $\deg(v)$ for its
{\em degree} $|N(v)|$.
A {\em walk} in the graph is a sequence $W = (v_0, v_1, ..., v_k)$ of vertices such that $v_iv_{i+1}$ is an edge in $G$ for each $i = 0, \ldots, k-1$.
For $i<j$, we write $v_iWv_j$ for the {\em segment} $(v_i, \ldots, v_j)$ of the walk $W$. We will also write $u_0W_1u_1W_2...W_tu_t$ for the walk obtained by pasting together the segements $u_0W_1u_1, \ldots, u_{t-1}W_tu_t$ of the walks $W_1, \ldots, W_t$ at their terminal vertices.
We define a {\em path} to be a walk without repeated vertices, and if a path has $k+1$ vertices we say that its \emph{length} is $k$.
If a graph is connected, there is a path between any two vertices and the {\em distance} between two vertices is the length of any shortest
path between them. For terms and notation that are not defined here, we refer to the standard text \cite{BM}.

\begin{defi}
Let  $(V, \rho)$ be a metric space and let $L$ be a line of $(V, \rho)$. The {\em generator graph of $L$}, denoted $H(L) = (V, E(L))$,
is the graph on vertex set $V$ with edge set
\[ E(L) = \left\{ \{a, b\} \in \binom{V}{2} : \ov{ab} = L \right\}; \]
and for any real number $\delta$, the graph $H_\delta(L) = (V, E_\delta(L))$ is the subgraph of $H(L)$ with edge set
\[ E_\delta(L) = \left\{ \{a, b\} \in \binom{V}{2} : \rho(a, b) = \delta, \ov{ab} = L \right\}.\]
\end{defi}

So the edges of $H(L)$ are those pairs that generate $L$ and the edges of $H_\delta(L)$ are those pairs that generate $L$ and are at distance $\delta$.
It is clear that $(E(L) : L \text{ is a line of } (V, \rho))$ is a partition of $\binom{V}{2}$, so we have $\sum |E(L)| = \binom{|V|}{2}$ where the sum is taken over all lines $L$ of $(V, \rho)$.
For any $\delta$, the sum of the number of edges in $H_\delta(L)$ over all lines $L$ equals the number of pairs of points in the metric space at distance $\delta$.

The following lemma is easy and we omit the proof.

\begin{lem}\label{lem.no_dd} Suppose $(V, \rho)$ is a metric space with diameter $D$ and $\delta > D / 2$. If there are three distinct points $a, b, c \in V$ such that $\rho(a, b) = \rho(a, c) = \delta$,
then $\{a,b,c\}$ is not collinear. In particular, for any line $L$, the generator graph $H_\delta(L)$ has
no vertex with degree greater than 1.
\end{lem}

Before stating the main theorem of this section, we give a short sketch of its proof; recall that we are giving a lower bound on the number of lines in a metric space in which the distance function takes a bounded number of distinct values.
Let $M = (V, \rho)$ be such a metric space on $n$ points. We first prove that either we have our desired number of lines or $M$ contains a lot of (that is, on the order of $n^2$) pairs of points at distance $D/2$, where $D$ is the diameter of $M$.
This implies that there is a line $L$ such that the number of edges in the generator graph $H_{D/2}(L)$ is large (that is, the number of edges is of order $n$).
Then we prove some strong structural properties of $H_{D/2}(L)$ that permit us to find enough distinct lines in the metric subspace of $M$ induced by the vertices of degree at least one in $H_{D/2}(L)$.

\begin{thm}\label{thm.bounded_dists} If $M = (V, \rho)$ is an $n$-point metric space ($n \geq 2$) and $W = \{\rho(a, b) : a, b \in V\}$, then $M$ has at least $\frac{|V|}{5|W|}$ distinct lines.
\end{thm}

\begin{proof}
Let $n = |V|$, $w = |W|$, let $D$ be the diameter of $M$, and let $m$ be the number of lines in $M$. We may assume that $m < \frac{n}{5w}$.

For any $z \in V$ and $\delta \in W$, set $S(z, \delta) = \{ x \in V : \rho(x, z) = \delta \}$; that is, the set of points with distance $\delta$ to $z$.

\begin{claim}
$w \geq 3$.
\end{claim}

\begin{proofclaim}
If $w = 1$, all pairs of points are at distance $D$ so no triple is collinear, and we have $m \geq \binom{n}{2} > n/5$.

Suppose next that $w = 2$; we shall show that $m \geq n/10$. We can write $W = \{\alpha, \beta\}$ for some two real numbers $\alpha$ and $\beta$. It is straightforward to check that unless $\alpha = 2\beta$ or $\beta = 2\alpha$, no triples are collinear and we again have $m \geq \binom{n}{2} > n/10$. So we may assume $W = \{1, 2\}$.

Consider any $V' \subseteq V$ where $|V'| \geq n/5 + 2$ and let $M' = (V', \rho)$ be the subspace of $M$ on ground set $V'$. For any $x, y \in V'$, the line generated by $x$ and $y$ in $M'$ is exactly $V' \cap \ov{xy}$ (that is, the intersection of $V'$ with the line they generate in $M$). Hence $M'$ has fewer than $n / 10$ lines, since $M$ does. But then it follows from Theorem~\ref{thm.12metric} applied to $M'$ that $M'$ has a universal line, so there exist $x, y \in V'$ such that $V' \subseteq \ov{xy}$.
Suppose that such $x, y$ exist with $\rho(x, y) = 1$. Then for each point $z \in V' \setminus \{x, y\}$, we must have $\rho(x, z) = 2$ or $\rho(y, z) = 2$. Thus, we may assume, by symmetry, that $|S(x, 2)| \geq n/10$. It follows that $\{ \ov{xz} : z \in S(x, 2) \}$ is a set of at least $n/10$ distinct lines.
So we may assume that any set $V'$ with $|V'| \geq n/5 + 2$ contains a pair $x, y$ such that $V' \subseteq \ov{xy}$ and $\rho(x, y) = 2$. Note that for any $z \in V' \setminus \{x, y\}$, we have $\rho(x, z) = \rho(y, z) = 1$.

Now, start from $V_0 = V$, and as long as $|V_i| \geq n/5 + 2$ (this is true whenever $i < t = \lfloor 2n/5 \rfloor - 1$), we find $x_i, y_i \in V_i$ such that both are at distance 1 to all points in $V_{i+1} = V_{i} \setminus \{x_i, y_i\}$. Then the pairwise distances among $\{x_0, x_1, ..., x_t\}$ are all 1, so the $\binom{t+1}{2}$ pairs of points in this set generate distinct lines. It is easy to check $\binom{t+1}{2} \geq n/10$ when $n > 10$. When $n \leq 10$, our claim is trivial because $m \geq 1$.
\end{proofclaim}

The next claim is immediate and will be used often.

\begin{claim}\label{lem.d_do2_flip}
 If three points $\{a, b, c\}$ are collinear, $\rho(a, b) = D/2$ and $\rho(b, c) \in \{D/2, D\}$, then $\rho(a, c) = 3D/2 - \rho(b, c)$.
\end{claim}

The next two claims show that we may assume that there are a lot of pairs of points at distance $D/2$.

\begin{claim}~\label{lem.pair_do2}
If $a, b \in V$ are points with $\rho(a, b) = D$, then the set $Y = \{x \in V : \rho(a, x) = \rho(b, x) = D/2\}$ has size at least $3n/5$.
\end{claim}

\begin{proofclaim}
Let $X_a = \{ x \in V : \rho(a, x) > D / 2\}$, $X_b = \{ x \in V : \rho(b, x) > D / 2\}$.
So $Y = V \setminus (X_a \cup X_b)$.
Hence, in order to prove that $|Y| \ge 3n/5$, it is enough to prove that both $|X_a|$ and $|X_b|$ are less than $n / 5$.

Assume that $|X_a| \geq n / 5$. Then there is a number $\delta \in W$ such that $\delta > D/2$ and $|S(a, \delta)| \geq n / (5 w)$.
By Lemma~\ref{lem.no_dd}, for any distinct elements $z_1$, $z_2$ of $S(a, \delta)$, the three points $\{a, z_1, z_2\}$ are not collinear. So the lines of the form $\ov{az}$ with $z \in S(a, \delta)$ are all pairwise distinct, which gives $n/(5w)$ distinct lines, contradicting the assumption that $m < n/(5w)$.
The symmetric argument shows that $|X_b| < n/5$.
\end{proofclaim}

\begin{claim}\label{lem.many_do2}
$M$ has at least $4n^2/25$ pairs of points at distance $D/2$.
\end{claim}

\begin{proofclaim} Pick any two points $a$ and $b$ such that $\rho(a, b) = D$, and consider the set $Y = \{x \in V : \rho(a, x) = \rho(b, x) = D/2\}$. By Claim \ref{lem.pair_do2} we have $|Y| \geq 3n / 5$.

Partition $Y$ into $s$ parts $(Y_1, \dots, Y_s)$ where $z_1,z_2$ are in the same part if and only if $\ov{az_1} = \ov{az_2}$.
So $s \leq m < n/(5w)$ and thus at least $3n / 5 - n / (5w)$ points in $Y$ are in some part $Y_i$ of size at least two.
If $z_1$ and $z_2$ are two points in the same part $Y_i$, then since $\ov{az_1} = \ov{az_2}$, $\{a, z_1, z_2\}$ is collinear; then since $\rho(a, z_1) = \rho(a, z_2) = D / 2$, we have $\rho(z_1, z_2) = D$.
So, the set
\[
A = \{ x \in V : \exists y \in V, \rho(x, y) = D\}
\]
has size $|A| \geq 3n / 5 - n / (5w)$, which is at least $8n / 15$ because $w \geq 3$.
By Claim~\ref{lem.pair_do2}, for any point $a \in A$, we have $|S(a, D/2) | \ge 3n / 5$.
Counting pairs with distance $D/2$ of which at least one element is in $A$,
we have at least $ \frac{1}{2} \frac{8n}{15} \frac{3n}{5} = \frac{4n^2}{25}$ pairs.
\end{proofclaim}

\begin{claim}\label{lem.parity}
Let $L$ be a line of $M$ and $P = (x_1, ..., x_t)$ a path in $H_{D/2}(L)$. Let $x$ be a point in $L$ but not in $P$.
If $\rho(x, x_1) \in \{D , D/2\}$, then
$\rho(x, x_t) = \rho(x, x_1)$ when $t$ is odd, and $\rho(x, x_t) = 3D / 2 - \rho(x, x_1)$ when $t$ is even.
\end{claim}

\begin{proofclaim}
We proceed by induction on $t$.
The base case $t = 1$ is trivial. Now suppose the claim holds for paths of length less than $t - 1$ and that $t \geq 2$. By the induction  hypothesis $\rho(x, x_{t-1}) = 3D / 2 - \rho(x, x_1)$ if $t$ is odd and $\rho(x, x_{t-1})  =\rho(x, x_1)$ if $t$ is even.
Since $x \in L = \ov{x_{t-1}x_t}$ and $\rho(x_{t-1}, x_t) = D/2$, by Claim~\ref{lem.d_do2_flip} we have
$\rho(x, x_t) = \rho(x, x_1)$ if $t$ is odd and $\rho(x, x_t) = 3D / 2 - \rho(x, x_1)$ if $t$ is even.
\end{proofclaim}

\begin{claim}\label{lem.bipartite}
Let $L$ be a line in $M$. If $C$ is a connected component of $H_{D/2}(L)$ with at least two vertices, then $V(C)$
can be partitioned into sets $A$ and $B$ such that

(a) all the edges in the component $C$ are between $A$ and $B$ (i.e. $H_{D/2}(L)$ is a bipartite graph);

(b) for any distinct vertices $a_1$ and $a_2$ in $A$, $\rho(a_1, a_2) = D$;  for any distinct vertices $b_1$ and $b_2$ in $B$, $\rho(b_1, b_2) = D$;
for any $a \in A$ and $b \in B$, $\rho(a, b) = D / 2$.
\end{claim}

\begin{proofclaim}
If there is an odd cycle in $C$, then $C$ has two vertices $x$ and $z$ and paths $(x, x_1, \ldots, x_t)$ and $(x, y_1, \ldots, y_s)$ such that $z = x_t = y_s$, $t$ is odd, and $s$ is even. Then by Claim~\ref{lem.parity}, we have both $\rho(x, z) = D/2$ and $\rho(x,z) = 3D/2$, a contradiction. This proves (a).
Pick any path $(x, x_1, x_2, ..., x_t)$ in $C$.
Since $\{x, x_1\}$ is an edge of $H_{D/2}(L)$, we have $\rho(x, x_1) = D / 2$ and thus, by Claim~\ref{lem.parity}, we have $\rho(x, x_t) = D / 2$ if $t$ is odd and $\rho(x, x_t) = D$ if $t$ is even; (b) follows easily.
\end{proofclaim}

\begin{claim}\label{lem.only_d_do2}
Let $L$ be a line and let $H' = (V', E')$ be a subgraph of $H_{D/2}(L)$ such that every vertex in $V'$ has degree at least two in $H'$.
If $x$ and $y$ are distinct vertices of $H'$, then $\rho(x, y) \in \{D/2, D\}$.
\end{claim}

\begin{proofclaim}
Note that $V' \subseteq L$.
We first prove that for any distinct points $x_0$ and $y_0$ in $V'$, $\rho(x_0, y_0) \leq D/2$ or $\rho(x_0, y_0) = D$.
If $x_0$ and $y_0$ are in the same component of $H'$, then by part (b) of Claim~\ref{lem.bipartite} we are done. Otherwise,
suppose that $\rho(x_0, y_0) = \delta$ for some $\delta$ satisfying $D/2 < \delta < D$.
Since every vertex in $H'$ has degree at least $2$, there are points $z_1$ and $z_2$ in $V'$ such that $\{x_0,z_1\}$ and $\{x_0,z_2\}$ are edges of $H'$. So in particular $\ov{x_0z_1} = \ov{x_0z_2} = L$ and $\rho(x_0,z_1)=\rho(x_0,z_2)=D/2$.
Also note that $\rho(z_1, z_2) = D$ by Claim~\ref{lem.bipartite}.
Since $y_0 \in L$, $\{x_0, z_1, y_0\}$ is collinear and thus we must have $\rho(y_0, z_1) = \delta - D / 2$ and similarly $\rho(y_0, z_2) = \delta - D / 2$.
So, we have  $D= \rho(z_1,z_2) \le \rho(z_1,y_0) + \rho(y_0,z_2) = 2\delta - D < D$, a contradiction.
So the distance between any two points in $V'$ is either at most $D/2$, or equal to $D$.

It remains to show that for any two points $x$ and $y$ in $V'$, we do not have $\rho(x, y) < D / 2$.
If $x$ and $y$ are in the same connected component of $H_{D/2}(L)$, then we are done by part (b) of Claim~\ref{lem.bipartite}.
So we may assume $x$ and $y$ are not in the same connected component of $H_{D/2}(L)$.
Let $\delta' = \rho(x, y)$, so $0 < \delta' < D/2$.
We let $z_1$ and $z_2$ be points such that $\{x, z_1\}$ and $\{x, z_2\}$ are edges of $H'$.
Since $\{x, y, z_1\}$ is collinear,
$\rho(z_1, y) = D / 2 - \delta'$ or $\rho(z_1, y) = D / 2 + \delta'$.
But we have shown that no distance between two points in $V'$ is in the interval $(D/2, D)$, so $\rho(z_1, y)$ cannot equal $D/2 + \delta'$.
Hence we have $\rho(z_1, y) = D / 2 - \delta'$.
Similarly, $\rho(z_2, y) = D / 2 - \delta'$. So $\rho(z_1, z_2) \leq D - 2\delta' < D$, a contradiction.
\end{proofclaim}

We are now armed to finish the proof.
By Claim~\ref{lem.many_do2}, there are at least $4n^2/25$ pairs of points at distance $D/2$.
It follows that there is some line $L$ such that $H_{D/2}(L)$ has $12n/5$ edges, for otherwise $M$ would have at least $n/15 \geq n/(5w)$ lines. We construct a subgraph $H'$ of $H_{D/2}(L)$ by repeatedly deleting vertices of degree less than three until we obtain a subgraph of minimum degree at least three. We can delete at most $n$ vertices, which means we delete at most $2n$ edges, so $H'$ has at least $2n/5$ edges.

Let $C_1$, $C_2$, ..., $C_k$ be the connected components of $H'$. For each $i = 1, \ldots, k$, it follows from Claim~\ref{lem.bipartite} that $C_i$ is a bipartite graph with vertex partition $(A_i, B_i)$ such that no edge of $C_i$ is contained in $A_i$ or in $B_i$.
Let $a_i = |A_i|$ and $b_i = |B_i|$. We may assume that $a_i \geq b_i$ and that $a_1 \geq a_2 \geq ... \geq a_k$. Note that $a_k \geq 3$ by the definition of $H'$.
In the rest of the proof we are going to show that lines generated by pairs of vertices that are in the same set $A_i$ are all different.

\begin{claim}
 Let $i$ and $j$ be distinct elements of $\{1, \ldots, k\}$ and let $a, b$ be distinct elements of $A_i$.
 If $c, d$ are distinct points of $A_i$ such that $\{a, b\} \neq \{c, d\}$, then $\ov{ab} \neq \ov{cd}$.
 If $c, d$ are distinct points in $A_j$, then $\ov{ab} \neq \ov{cd}$.
\end{claim}

\begin{proofclaim}
It follows from part (b) of Claim~\ref{lem.bipartite} that, for any vertex $u \in A_i \setminus \{a, b\}$, we have $\rho(a, u) = \rho(b, u) = \rho(a, b) = D$, so $\ov{ab} \cap A_i = \{a, b\}$ and this implies that for any points $c, d \in A_i$, $\ov{cd} \neq \ov{ab}$ unless $\{c, d\} = \{a, b\}$.

Now consider any $v, w \in A_i$.
By Claim~\ref{lem.only_d_do2}, we have $\rho(v, c) \in \{D/2, D\}$, and by repetead applications of Claim~\ref{lem.parity}, we have $\rho(v, c) = \rho(c, w) = \rho(w, d) = \rho(d, v)$.
This means that there is some $\alpha \in \{D/2, D\}$ such that $\rho(z, c) = \rho(z, d) = \alpha$ for all $z \in A_i$.
Therefore, the line $\ov{cd}$ is either disjoint from $A_i$ or contains all of $A_i$, because the distance from $c$ to every vertex in $A_i$ is the same and the distance from $d$ to every vertex in $A_i$ is the same.
On the other hand, $\ov{ab} \cap A_i = \{a, b\}$ as we observed in the last paragraph. Because $a_i \geq 3$, it follows that $\ov{ab} \neq \ov{cd}$.
\end{proofclaim}

This means that the number of lines in $M$ is at least
\[
\sum_{i=1}^k \binom{a_i}{2} \geq \sum_{i=1}^k \frac{a_i^2}{3} \geq \sum_{i=1}^k \frac{a_i b_i}{3} \geq \frac{2n}{15},
\]
where the first inequality follows from the fact that $a_i - 1 \geq \frac{2}{3}a_i$ since $a_i \geq 3$, and the last inequality follows from the fact that $2n/5 \leq |E(H')| \leq \sum_{i=1}^k a_i b_i$.
Since $w \geq 3$, we have $2n/15 \geq 2n/5w$, which finishes the proof.
\end{proof}

\section{When each nonzero distance equals 1, 2 or 3}\label{sect.123}

By a {\em $k$-metric space}, we mean a metric space in which the distance between any two points is an integer less than or equal to $k$.
This includes all metric spaces induced by graphs with diameter $k$.
Theorem \ref{thm.bounded_dists} immediately implies

\begin{cor} For any positive integer $k$, every $k$-metric space on $n$ points has at least $n/(5k)$ lines.
\end{cor}

Conjecture~\ref{conj.cc} itself is known to hold for $2$-metric spaces; as we stated in Theorem~\ref{thm.12metric}.
Chiniforooshan and Chv\' atal \cite{Chini_Chvatal} proved that, asymptotically, there are even more lines than this:

\begin{thm}\label{thm.2metric} (\cite{Chini_Chvatal})
The smallest number $h(n)$ of lines in a 2-metric space on $n$ points satisfies the inequalities
\[(1 + o(1))\alpha n^{4/3} \leq h(n) \leq (1 + o(1))\beta n^{4/3}\]
with $\alpha=2^{-7/3}$ and $\beta=3 \cdot 2^{-5/3}$.
\end{thm}

Here we use this theorem (and some ideas from the proof) to prove an $\Omega(n^{4/3})$ lower bound on the number of lines in $3$-metric spaces on $n$ points.

\begin{thm}\label{thm.3metric} The smallest number $t(n)$ of lines in a 3-metric space on $n$ points is $\Theta(n^{4/3})$.
\end{thm}

\begin{proof} Because any 2-metric space is itself a 3-metric space, the upper bound (that is, the fact that $t(n) \in O(n^{4/3})$) follows from Theorem \ref{thm.2metric}. We are now going to show $t(n) \in \Omega(n^{4/3})$.
Let $(V, \rho)$ be a 3-metric space with $|V| = n$, and consider the $\binom{n}{2}$ pairs of different points in $V$. We have three cases.

{\em Case 1.} At least $\binom{n}{2} / 3$ pairs are at distance 1. Then there is at least one point $a$ such that the set
$X_a = \{v \in V : \rho(a, v) = 1\}$
has size $|X_a| \geq (n-1)/3$. For any two points $x_1, x_2 \in X_a$, $\rho(x_1, x_2) \leq \rho(a, x_1) + \rho(a, x_2) = 2$.
So the metric space $(X_a, \rho)$ is a 2-metric space.
By Theorem \ref{thm.2metric}, there are at least $\Omega(n^{4/3})$ lines
in $(X_a, \rho)$. Note that for any $x_1, x_2 \in X_a$, the intersection of $X_a$ with  the line generated by $x_1$ and $x_2$ in $(V, \rho)$
is exactly the line they generate in $(X_a, \rho)$. So we have $\Omega(n^{4/3})$ lines in $(V, \rho)$.

{\em Case 2.} At least $\binom{n}{2} / 3$ pairs are at distance 2. Consider the lines generated by these pairs. If no line is generated by more than
$n^{2/3}$ such pairs then we have at least $\Omega(n^{4/3})$ lines. Otherwise, there is a line $L$ such that the generator graph $H_2(L)$ has $s > \lfloor n^{2/3} \rfloor$ edges.
By Lemma \ref{lem.no_dd}, $H_2(L)$ has no vertex of degree greater than one. We can enumerate the edges of $H_2(L)$ as $\{a_1, b_1\}, \ldots, \{a_s, b_s\}$. Since none of the distances $\rho(a_i, a_j)$ can be equal to 2 (otherwise $\{a_i,a_j,b_i\}$ is not collinear), we have $\rho(a_i, a_j) \in \{1,3\}$ for any $i \neq j$. By parity,
for any distinct $i$, $j$, $k$, $[a_ia_ja_k]$ does not hold. Thus
\[\ov{a_ia_j} \cap \{a_1, \dots, a_s\} = \{a_i,a_j\}.\]
We have $\binom{s}{2} \in \Omega(n^{4/3})$ distinct lines $\ov{a_ia_j}$ ($1 \le i < j\le s$).

{\em Case 3.} At least $\binom{n}{2} / 3$ pairs are at distance 3. Consider the lines generated by these pairs. If no line is generated by more than
$n^{2/3}$ such pairs then we have at least $\Omega(n^{4/3})$ lines. Otherwise, there is a line $L$ such that the generator graph $H_3(L)$ has
$s > \lfloor n^{2/3} \rfloor$ edges.
By Lemma \ref{lem.no_dd}, $H_3(L)$ has no vertex of degree greater than one and we call them $\{a_1,b_1\}, \dots, \{a_s,b_s\}$.
Let $X = \bigcup_{i=1}^{s}\{a_i, b_i\}$.

Observe that, for any $j \le s$, since $\rho(a_j,b_j)=3$ and $X \subseteq \ov{a_jb_j}$, we have that:
\begin{equation}\label{eq.3matching_gen}
\forall x \in X \setminus \{a_j, b_j\}, \text{ either } \rho(a_j, x) = 2 \text{ or } \rho(b_j, x) = 2.
\end{equation}

For any $1 \le i < j \le s$, set $x_{ij} = a_j$ if $\rho(a_i,a_j)=2$, otherwise set $x_{ij}=b_j$.
So, by (\ref{eq.3matching_gen}), $\rho(a_i,x_{ij})=2$.
Now, for each $1 \le i<j\le s$ set $L_{ij}=\ov{a_ix_{ij}}$.
Since every three points in $\{a_i, a_j, b_i, b_j\}$ are collinear, we clearly have $\{a_i,a_j,b_i,b_j\} \subseteq L_{ij}$.
On the other hand, consider any $r \in \{1, \ldots, k\} \setminus \{i, j\}$.
By (\ref{eq.3matching_gen}), one of $\rho(a_i,a_r)$ or $\rho(a_i,b_r)$ equals two and thus, by Lemma~\ref{lem.no_dd}, at least one of $a_r$ and $b_r$ is not in $L_{ij}$.
 Therefore, we conclude that all the lines $L_{ij}$ for $1 \le i<j\le s$ are pairwise distinct. This gives $\binom{s}{2} \in \Omega(n^{4/3})$ distinct lines.
\end{proof}

\section{Pairs generating the same line}\label{sect.pairs}

In this section, we prove some facts about the relationships between pairs of points in a pseudometric betweenness that generate the same line.
These technical results will be used in the next section to prove bounds on the number of lines in graph metrics; they also seem likely to be of further use in proving theorems about lines in pseudometric betweennesses and in metric spaces.

We start with an easy lemma used as a tool in several proofs.

\begin{lem}\label{fact.center} Let $\mc B$ be a pseudometric betweenness on a set $V$ and let $a$, $b$, $c$ and $x$ be four distinct points in $V$. If $[axb]$, $[bxc]$, and $[cxa]$, then $\{a, b, c\}$ is not collinear.
\end{lem}

\begin{proof}
We may assume, without loss of generality, that $[abc]$.
Since $[axb]$, we have $[axbc]$, a contradiction to $[bxc]$.
\end{proof}

\begin{defi}Let $\mc B$ be a pseudometric betweenness on a set $V$.  For two distinct points $a$ and $b$ in $V$,
define
\[ I(a, b) = \{x \in V : [axb]\} \text{ and } O(a, b) = \{x \in V : [xab] \mbox{ or } [abx]\}. \]
\end{defi}

Thus the line $\ov{ab}$ is
\[ \ov{ab} = \{a, b\} \cup I(a, b) \cup O(a, b).\]

\begin{defi} For four distinct points $a$, $b$, $c$, and $d$ in a pseudometric betweenness, we call the $4$-tuple $(a, b, c, d)$ a \emph{parallelogram} if
\[[abc],\,  [bcd],\,  [cda],\, \mbox{and } [dab].\]
\end{defi}

It is clear from the definition that, if a sequence is a parallelogram, it remains a parallelogram if we reverse it, or if we move the first element of the sequence to the end.

\begin{defi} Let $a$, $b$, $c$, and $d$ be four distinct points in a pseudometric betweenness. We call the two pairs $\{a, b\}$ and $\{c, d\}$ {\em parallel} if $(a, b, c, d)$ or $(a, b, d, c)$ is a parallelogram; we call $\{a, b\}$ and $\{c, d\}$ {\em antipodal} if $(a, c, b, d)$ is a parallelogram.
\end{defi}

\begin{defi}
Let $\mc B$ be a pseudometric betweenness on a set $V$,  and let $a$, $b$, $x$ and $y$ be  points in $V$ such that $a \neq b$, $x \neq y$, $\{a, b\} \neq \{x, y\}$ and  $\ov{ab} = \ov{xy}$. We say that the pairs $\{a,b\}$ and $\{x,y\}$ are
\begin{itemize}
\item $\alpha$-related if $\{a,b,x,y\}$ (a set of size $3$ or $4$) is geodesic;
\item $\beta$-related if $\{a, b\}$ and $\{x, y\}$ are parallel and $I(a, b) = I(x, y)=\emptyset$;
\item  $\gamma$-related if  $\{a, b\}$ and $\{x, y\}$ are antipodal and $O(a, b) = O(x, y)=\emptyset$.
\end{itemize}
Moreover, for a pair of points $\{q,r\}$, we call $\{q, r\}$ a $\beta$-pair if there is  another pair $\{s, t\}$ that is $\beta$-related to it, and we call $\{q, r\}$ a $\gamma$-pair if there is  another pair $\{s, t\}$ that is $\gamma$-related to it.
\end{defi}

The following lemma gives a very strong property satisfied by any two pairs of points defining the same line.
Even though it is valid for any pseudometric betweenness, we have only been able to use it efficiently in graph metrics.
Anyway, we give the proof in its whole generality, hoping it will be useful in future work.

\begin{lem} \label{lem.betagamma}
Let $\mc B$ be a pseudometric betweenness on a set $V$,  and let $a$, $b$, $x$ and $y$ be  points in $V$ such that $a \neq b$, $x \neq y$ and $\{a, b\} \neq \{x, y\}$. If $\ov{ab} = \ov{xy}$, then $\{a,b\}$ and $\{x,y\}$ are $\alpha$-related, $\beta$-related or $\gamma$-related.
\end{lem}

\begin{proof}
If two of the points in $\{a, b, x, y\}$ are the same, say $a = y$, then trivially $\{a,b\}$ and $\{a,x\}$ are $\alpha$-related.
So we may assume $\{a, b, x, y\}$ consists of four points and is not geodesic (otherwise $\{a,b\}$ and $\{x,y\}$ are $\alpha$-related and we are done).
Note that because $\ov{ab} = \ov{xy}$, any three among the four points are collinear.
\medskip

{\em Case 1:} One of $[axb]$, $[ayb]$, $[xay]$ or $[xby]$ holds.
In this case, we show that $\{a,b\}$ and $\{x,y\}$ are $\gamma$-related.

We first show that any one of these would imply the other three. Without loss of generality, we assume $[axb]$.
Observe that $\{a,b,y\}$ is collinear.
Since $[axb]$, if $[aby]$, then $[axby]$ and if $[bay]$, then $[bxay]$, a contradiction in both cases.
So we may assume $[ayb]$.
Next, we use the fact that $\{a, x, y\}$ is collinear.
If $[ayx]$, we have $[ayxb]$, and if $[axy]$, since $[ayb]$, we have $[axyb]$, a contradiction in both cases.
So we may assume that $[xay]$.
Similarly, we have $[xby]$.

It remains to prove that $O(a,b)=O(x,y)=\emptyset$.
Suppose not; without loss of generality we may assume that there exists a point $u \in V$ such that $[abu]$.
Observe that $u \in \ov{ab}=\ov{xy}$ and thus, in particular, $\{x,y,u\}$ is collinear.
Since $[axb]$ and $[ayb]$, we have $[axbu]$ and $[aybu]$.
So, we have $[xbu]$, $[ybx]$, $[uby]$   and thus, by Lemma~\ref{fact.center}, $\{x,y,u\}$ is not collinear, a contradiction.
\medskip

{\em Case 2:} None of $[axb]$, $[ayb]$, $[xay]$ or $[xby]$ holds. In this case we prove that $\{a,b\}$ and $\{x,y\}$ are $\beta$-related.

We have that $\{a,b\} \subseteq O(x,y)$ and $\{x,y\} \subseteq O(a,b)$.
Assume without loss of generality that $[xya]$.
We prove first that $[yab]$, $[abx]$ and $[bxy]$.

We have that $\{x, a, b\}$ is collinear.
Since $[ayx]$, if $[bax]$, then $[bayx]$, a contradiction.
Hence, since $x \in O(a, b)$, we have $[abx]$.
We have that $\{x,y,b\}$ is collinear.
If $[byx]$, then $[abyx]$, a contradiction.
Hence, since $b \in O(x, y)$, we have $[bxy]$.
Finally, we have that $\{a,b,y\}$ is collinear.
Since $[bxy]$, if $[aby]$, then $[abxy]$, a contradiction.
Hence, since $y \in O(a, b)$, we have $[yab]$.

So, $(x, y, a, b)$ is a parallelogram and it remains to show that $I(a,b)=I(x,y)=\emptyset$.
Suppose not; without loss of generality we may assume that there exists a point $u \in V$ such that $[aub]$.
Hence, since $[abx]$ and $[bay]$, we have $[aubx]$ and $[buay]$.
Since $u \in \ov{ab}=\ov{xy}$, the points $\{x, y, u\}$ are collinear. So either $[xyu]$, $[yxu]$, or $[xuy]$ holds.
If $[xyu]$, since $[yau]$, we have $[xyau]$, a contradiction to $[aux]$.
If $[yxu]$, since $[yub]$, we have $[yxub]$, a contradiction to $[ubx]$.
If $[xuy]$, since $[uay]$, we have $[xuay]$, a contradiction to $[xya]$.
\end{proof}

\begin{lem}\label{lem.gamma_clique} In any pseudometric betweenness,

(a) no pair of points is both a $\beta$-pair and a $\gamma$-pair;

(b) any two $\gamma$-pairs that generate the same line are disjoint;

(c) if $\{a, b\}$ and $\{c, d\}$ are distinct $\gamma$-pairs with $\ov{ab} = \ov{cd}$, then they are $\gamma$-related.
\end{lem}

\begin{proof}
(a) If a pair $\{a, b\}$ is a $\gamma$-pair, then $I(a, b) \neq \emptyset$ and therefore it cannot be a $\beta$-pair.

(b) Suppose that $\{a, b\}$ and $\{a, c\}$ are $\gamma$-pairs such that $\ov{ab} = \ov{ac}$. Then $\{a, b, c\}$ is collinear and one of $[abc]$, $[acb]$ and $[bac]$ holds. But then $O(a, b)$ and $O(a, c)$ cannot both be empty, contradicting the definition of a $\gamma$-pair.

(c) Suppose $\{a, b\}$ and $\{c, d\}$ are distinct $\gamma$-pairs with $\ov{ab} = \ov{cd}$. These pairs are not $\beta$-related by (a). It follows from (b) that $a$, $b$, $c$ and $d$ are four distinct points. If $\{a, b, c, d\}$ is geodesic, then, since $O(a, b) = \emptyset$, we have $[acdb]$ or $[adcb]$. In either case, $O(c,d) \neq \emptyset$, a contradiction. Thus, by Lemma~\ref{lem.betagamma}, $\{a, b\}$ and $\{c, d\}$ are $\gamma$-related.
\end{proof}

\begin{lem}\label{lem.gamma_no_mid} Let $\mc B$ be a pseudometric betweenness in which $\{a, b\}$, $\{u, v\}$ and $\{x, y\}$ are three pairs of points that are pairwise $\gamma$-related.
If  $[axu]$, then $\{a,y,u\}$ is not collinear.
\end{lem}

\begin{proof}
Assume that $[axu]$ holds. Since $[aub]$, we have $[xub]$ and thus, since $[xby]$, we have $[uby]$.

By way of contradiction, assume that $\{a, y, u\}$ is collinear. Then either $[auy]$, $[ayu]$ or $[uay]$ holds.
If $[auy]$, then since $[ayb]$, we have $[uyb]$, contradicting $[uby]$.
If $[ayu]$, then since $[aub]$, we have $[yub]$, also contradicting $[uby]$.
Finally, if $[yau]$, then since $[axu]$, we have $[yxu]$, contradicting the fact that $\{u, v\}$ and $\{x, y\}$ are $\gamma$-related.
\end{proof}

The final lemma of this section is particular to metric spaces.

\begin{lem}
Let $a, b, c, d$ be distinct points in a metric space. The sequence $(a, b, c, d)$ is a parallelogram if and only if
\begin{itemize}
 \item $\rho(a, b) = \rho(c, d)$,
 \item $\rho(a, d) = \rho(b, c)$, and
 \item $\rho(a, c) = \rho(b, d)= \rho(a, b) + \rho(b, c)$.
\end{itemize}
\end{lem}

\begin{proof} The {\em if} direction is obvious. For the {\em only if} direction,
$[abc]$ and $[cda]$ imply that
\begin{equation}\label{eq.para001}
\rho(a, b) + \rho(b, c) = \rho(c, d) + \rho(d, a);
\end{equation}
$[bcd]$ and $[dab]$ imply that
\begin{equation}\label{eq.para002}
\rho(b, c) + \rho(c, d) = \rho(d, a) + \rho(a, b);
\end{equation}
(\ref{eq.para001}) and (\ref{eq.para002}) imply that $\rho(a, b) = \rho(c, d)$ and $\rho(a, d) = \rho(b, c)$; and that in turn implies
\[ \rho(a, c) = \rho(a, b) + \rho(b, c)  = \rho(b, c) + \rho(c, d) = \rho(b, d). \]
\end{proof}

%
%

\section{Bounds on graph metrics}\label{sect.graph}

The goal of this section is to prove Theorem~\ref{thm.d_graph}, which asserts that there is a constant $c$ such that any $n$-point metric space induced by a graph of diameter $D$ has at least $c(n/D)^{4/3}$ lines.
The proof is based on the fact that any two pairs of vertices are either $\alpha$-related, $\beta$-related or $\gamma$-related.

Here is a short sketch of the proof.
Let $(V,\rho)$ be a metric space induced by a graph of diameter $D$ with no universal line. We want to prove that it has $\Omega((n/D)^{4/3})$ lines.
So we may assume that there is a line $L$ generated by at least $\Omega(n^{2/3} D^{4/3})$ different pairs of points.
Moreover, since the graph is of diameter $D$, there exists a distance $d$ such that at least $\Omega(n^{2/3} D^{1/3})$ of the pairs generating $L$ are pairs at distance $d$, that is, $|E_d(L)| = \Omega(n^{2/3} D^{1/3})$.

At this point, we can prove that there is a subset of at least $\Omega(n^{2/3} D^{1/3})$ of these pairs that are either pairwise $\alpha$-related or are pairwise $\gamma$-related.
The former case is handle by Lemma~\ref{lem.graph_alpha_quad}, the second on by Lemma~\ref{lem.graph_gamma_quad}.

Before proving the two main lemmas, we prove the following lemma that resembles the intermediate value theorem.

\begin{lem}\label{lem.med} Let $(V, \rho)$ be a metric space induced by a graph $G$, let $a, b \in V$ such that $\rho(a, b) \ge 2$, and let $W$ be a walk in $G$ from $a$ to $b$.
There exists a vertex $u \in W \sm \{a,b\}$ such that $u \not\in O(a, b)$.
\end{lem}

\begin{proof}
Since $W$ is a walk from $a$ to $b$, there exists an induced path $P = (a,u_1, \dots, u_k, b)$ from $a$ to $b$ where $\{u_1, \dots, u_k\}$ are all in $W \sm \{a,b\}$.
For contradiction, we may assume that, for any $1 \le i \le k$, $[u_iab]$ or $[abu_i]$.

Since $\rho(a, u_1) = 1$, we cannot have $[abu_1]$, so $[u_1ab]$ holds. Similarly, we have $[abu_k]$.
So there exists a maximum integer $j$ such that $[u_jab]$, and $j < k$. We have $[abu_{j+1}]$.
Hence, we have:
\[
\begin{split}
&\rho(u_j, a)  + \rho(a,b) = \rho(u_j,b)\le \rho(u_j,u_{j+1})+\rho(u_{j+1},b)=1+\rho(u_{j+1},b)\\
&\rho(a,b)  + \rho(b,u_{j+1}) = \rho(a,u_{j+1})\le \rho(a, u_j)+\rho(u_j,u_{j+1})=\rho(a, u_{j})+1
\end{split}
\]
Which implies that $\rho(a,b)-1\le 1-\rho(a,b)$ and thus $\rho(a,b)\le 1$, a contradiction.
\end{proof}

\begin{lem}\label{lem.graph_alpha_quad} Let $G$ be a graph with diameter $D$ and $(V, \rho)$ the corresponding graph metric.
Let $Q$ and $l \geq 2$ be positive integers.
Let $\{a_1, b_1\}, \ldots, \{a_Q, b_Q\}$ be $Q$ pairs of vertices such that, for any $1 \le i \neq j \le Q$, the following hold:
\begin{itemize}
\item $\ov{a_ib_i}=\ov{a_jb_j}$,
\item $\{a_i,b_i\}$ and $\{a_j,b_j\}$ are $\alpha$-related and
\item $\rho(a_i,b_i)=\rho(a_j,b_j) = \ell$.
\end{itemize}
Then the number of lines in $(V,\rho)$ is at least $(1/2 - o(1)) (Q / D)^2$ as $Q/D \rightarrow \infty$.
\end{lem}

\begin{proof}
Define $A = \{ k : a_1 \notin I(a_k, b_k) \}$.

\begin{claim}\label{eq.claim001}
$|A| > Q - D$.
\end{claim}

\begin{proofclaim}
This is equivalent to proving that $|B| < D$ where $B = \{k: [a_ka_1b_k] \}$.
For any $k \in B$, since $[a_ka_1b_k]$ and  $\rho(a_1, b_1) = \rho(a_k, b_k)$, the point $b_1$ is distinct from $a_k$ and $b_k$. Moreover, since the pairs $\{a_1, b_1\}$ and $\{a_k, b_k\}$ are $\alpha$-related, we have either $[a_ka_1b_kb_1]$ or $[b_1a_ka_1b_k]$.
By symmetry between $a_k$ and $b_k$, we may swap the labels $a_k$ and $b_k$ if necessary so that $[a_ka_1b_kb_1]$ for every $k \in B$.
\medskip

Let $i, j$ be distinct elements of $B$ and let us prove that $\rho(b_1, b_i) \neq \rho(b_1, b_j)$.
For a contradiction, we assume that $\rho(b_1,b_i) = \rho(b_1,b_j)$.
Since  $[a_ia_1b_1]$ and $[a_1b_jb_1]$, we have $[a_ib_jb_1]$.
So
\begin{align*}
\rho(a_i,b_j)& = \rho(a_i,b_1)-\rho(b_1,b_j)\\
& = \rho(a_i,b_1)-\rho(b_1,b_i)\\
& = \rho(a_i,b_i)=\ell \text{ (because }[a_ib_ib_1]).
\end{align*}
If $b_i \neq b_j$, then $\{a_i,b_i,b_j\}$ is collinear and  we must have $\rho(b_i,b_j)=\rho(b_i,a_i)+\rho(a_i,b_j)=2\ell$ .
But observe that $[a_1b_ib_1]$ and $[a_1b_jb_1]$, which imply that $\rho(b_i,a_1) < \rho(a_1, b_1) = \ell$ and  $\rho(a_1,b_j) < \rho(a_1, b_1) = \ell$. Therefore, $\rho(b_i,b_j)\le \rho(b_i,a_1)+\rho(a_1,b_j)<2\ell$, a contradiction.

Hence we may assume that $b_i = b_j$, and thus $a_i \neq a_j$.
Since $\{a_i, a_j, b_j\}$ is collinear, we must have $\rho(a_i, a_j) = \rho(a_i, b_j) + \rho(a_j, b_j) = 2\ell$.
But observe that $\rho(a_i, a_1) < \ell$ and $\rho(a_j, a_1) < \ell$ because $[a_ia_1b_i]$ and $[a_j a_1 b_j]$ and $\rho(a_i, b_i) = \rho(a_j, b_j) = \ell$. Therefore, $\rho(a_i, a_j) \le \rho(a_i, a_1) + \rho(a_1, a_j) < 2\ell$, a contradiction.

Hence, for any $i, j \in B$ with $i \neq j$, we have $\rho(b_1, b_i) \neq \rho(b_1, b_j)$.
Moreover, for each $k \in B$, since $[a_1b_kb_1]$ and $\rho(a_1, b_1) = \ell$, we have $\rho(b_1, b_k) \in \{1, \dots, \ell-1 \}$ (all distances are integral because $(V, \rho)$ is a graph metric). Hence $|B| \leq \ell - 1 < D$.
\end{proofclaim}

For any $k \in A$, may we  swap $a_k$ and $b_k$ if necessary so that either $[a_1a_kb_k]$ or $a_1 = a_k$ holds.
The distance $\rho(a_1, b_k)$ is a positive integer no more than $D$, so there is an integer $d \leq D$ such that the set
$A_d= \{ k \in A : \rho(a_1, b_k) = d \}$ has size $|A_d| \geq (Q-D) / D$.
The rest of the proof consists of showing that each $b_k$ for $k \in A_d$ is distinct and that each pair $\{b_i,b_j\}$ for $i, j \in A_d$ defines a distinct line.

The next claim is used as a tool.

\begin{claim}\label{tool}
Let $i, j$ be distinct elements of $A_d$.
For any vertex $x$, $[a_1xa_i]$ and $[a_jxb_j]$ cannot both hold.
\end{claim}

\begin{proofclaim}
Assume for contradiction that there is a vertex $x$ such that $[a_1xa_i]$ and $[a_jxb_j]$.
Since $[a_1a_ib_i]$, we have $[a_1xa_ib_i]$ and thus $\rho(a_1,x)<\rho(a_1,b_i)-\rho(a_i,b_i)=d-\ell$.
Moreover, since $[a_1a_jb_j]$, we have $[a_1a_jxb_j]$, and thus $\rho(a_1,x)>\rho(a_1,b_j)-\rho(a_j,b_j)=d-\ell$, a contradiction.
\end{proofclaim}

\begin{claim}\label{eq.claim002}
If $i, j$ are distinct elements of $A_d$, then $b_i \neq b_j$.
\end{claim}

\begin{proofclaim}
Let $i,j$ be distinct elements of $A_d$ and suppose for a contradiction that $b_i = b_j$.
Then $a_i\neq a_j$.
Since $\{a_i,b_i,a_j\}$ is collinear and $\rho(a_i,b_i)=\rho(a_j,b_i)=\ell$, we must have $[a_ib_ia_j]$.

Let $P_1$ be a shortest path from $b_i$ to $a_j$, let $P_2$ be a shortest path from $a_j$ to $a_1$ and finally let $P_3$ be a shortest path from $a_1$ to $a_i$.
So $W = b_i P_1 a_j P_2 a_1 P_3 a_i$ is a walk from $b_i$ to $a_i$.
Since $\rho(a_i,b_i) =\ell \ge 2$, by Lemma~\ref{lem.med}, there exists a vertex $u \in W \setminus \{a_i,b_i\}$ such that $u \notin O(a_i,b_i)$. We shall prove that such a $u$ does not exist, giving a contradiction.

It follows from $[a_ib_ia_j]$ and $[a_1a_ib_i]$ that $u \notin \{a_1,a_j\}$.
For any $x \in P_1 \sm \{b_i,a_j\}$, we have $[b_ixa_j]$ and since $[a_ib_ia_j]$, we have $[a_ib_ix]$, so $u \notin P_1 \sm \{b_i,a_j\}$.
For any $y \in P_2 \sm \{a_1,a_j\}$, we have $[a_1ya_j]$ and thus,  by Claim~\ref{tool}, $[a_iyb_i]$ does not hold, so  $u \notin P_2 \sm \{a_1,a_j\}$.
Finally, for any $z \in P_3 \sm \{a_1,a_i\}$, we have $[a_1za_i]$ and since $[a_1a_ib_i]$, we have $[za_ib_i]$ and thus $u \notin P_3 \sm \{a_1,a_i\}$.
\end{proofclaim}

\begin{claim}\label{eq.claim003}
If $i, j$ are distinct elements of $A_d$, then $I(a_i, b_i) \cap I(a_j, b_j) = \emptyset$, and $[a_ia_jb_i]$ does not hold.
\end{claim}

\begin{proofclaim}
Let $i \in A_d$ and let $u$ be a vertex such that $[a_iub_i]$.
Let $j  \in A_d \sm \{i\}$ and assume for contradiction that $u=a_j$ or $[a_jub_j]$.
We are going to show that, under these  conditions, $b_j \notin \ov{a_ib_i}$, a contradiction.

Since $[a_1a_ib_i]$,  we have $[a_1a_iub_i]$.
If $u=a_j$ we have $[a_1a_ia_jb_i]$ and since $[a_1a_jb_j]$, we have $[a_1a_ia_jb_j]$.
If $[a_jub_j]$, since  $[a_1a_jb_j]$, we have $[a_1a_jub_j]$ and thus $[a_1a_iub_j]$.
So, in both cases, we have $\rho(a_i,b_j) = \rho(a_1,b_j)-\rho(a_1,a_i)=d-(d-\ell)=\ell$ (indeed, since $[a_1a_ib_i]$, we have $\rho(a_1,a_i)=\rho(a_1,b_i)-\rho(a_i,b_i)=d-\ell$).

By Claim~\ref{eq.claim002}, $b_i \neq b_j$ and thus  $\{a_i,b_i,b_j\}$ is collinear.
Since $\rho(b_i,a_i)=\rho(a_i,b_j)=\ell$, we must have $\rho(b_i,b_j)=2\ell$.
But,  since $[a_iub_i]$, we have $\rho(u,b_i)<\rho(a_i,b_i)=\ell$ and similarly $\rho(u,b_j)<\ell$.
Hence, $\rho(b_i,b_j) \le \rho(b_i,u)+\rho(u,b_j) < 2 \ell$, a contradiction.
\end{proofclaim}

\begin{claim}\label{eq.claim004}
If $i, j$ are distinct elements of $A_d$, then $[b_ia_ib_j]$
\end{claim}

\begin{proofclaim}
Let $i, j$ be distinct elements of $A_d$.
Since $b_j \in \ov{a_ib_i}$, the triple $\{a_i,b_i,b_j\}$ is collinear.
If $[a_ib_jb_i]$, then since $[a_1a_ib_i]$, we have $[a_1a_ib_jb_i]$, which contradicts the fact that $\rho(a_1, b_i) = \rho(a_1, b_j)$.
So we may assume that $[a_ib_ib_j]$.

Let $P_1$ be a shortest path from $b_i$ to $b_j$, let $P_2$ be a shortest path from $b_j$ to $a_j$, let $P_3$ be a shortest path from $a_j$ to $a_1$, and finally let $P_4$ be a shortest path from $a_1$ to $a_i$.
So $W = b_i P_1 b_j P_2 a_j P_3 a_1 P_4 a_i$ is a walk from $b_i$ to $a_i$.
Since $\rho(a_i,b_i) =\ell \ge 2$, by Lemma~\ref{lem.med}, there exists a vertex $u \in W \setminus \{a_i,b_i\}$ such that $u \notin O(a_i,b_i)$. We shall prove that such a $u$ does not exist, giving a contradiction.

For any $x \in P_1 \sm \{b_i\}$, we have $[b_ixb_j]$, and since $[a_ib_ib_j]$, we have $[a_ib_ix]$, so $u \notin P_1 \sm \{b_i\}$.
For any $y \in P_2 \sm \{b_i\}$, $[b_jxa_j]$ or $x=a_j$. Hence $x \in \ov{a_jb_j}=\ov{a_ib_i}$ and by Claim~\ref{eq.claim003}, $x \notin I(a_i,b_i)$. So $x \in O(a_i,b_i)$ and $u \notin  P_2 \sm \{b_i\}$.

For any $z \in P_3 \sm \{a_1,a_j\}$, we have $[a_1za_j]$ and since $[a_1a_jb_j]$, we have $[za_jb_j]$ and thus $z \in \ov{a_jb_j}=\ov{a_ib_i}$. By Claim~\ref{tool}, $[a_jzb_j]$ does not hold and thus $z \in O(a_j,b_j)$.
Hence $u \notin  P_3 \sm \{a_1,a_j\}$.
Finally, for any $w \in P_4 \sm \{a_i\}$, we have $[a_1wa_i]$ and since $[a_1a_ib_i]$, we have $[wa_ib_i]$, so $u \notin P_4 \sm \{a_i\}$.
\end{proofclaim}

\begin{claim}\label{last}
For any distinct $i,j,k \in A_d$, $\{b_i,b_j,b_k\}$ is not collinear.
\end{claim}

\begin{proofclaim}
Let $i,j,k$ be three distinct elements of $A_d$.
By symmetry, it suffices to prove that $[b_ib_jb_k]$ does not hold.
By Claim~\ref{eq.claim004}, we have $[b_ja_ib_i]$ and $[b_ka_ib_i]$.
Hence
\[\rho(b_j,b_k) \leq \rho(b_j,a_i)+\rho(a_i,b_k) < \rho(b_j,b_i)+\rho(b_i,b_k).\]
Thus $[b_ib_jb_k]$ does not hold.
\end{proofclaim}

Set $C = \{b_k : k \in A_d \}$.
By Claim~\ref{last}, for any distinct elements $i, j$ of $A_d$, we have $\ov{b_ib_j} \cap C = \{b_i, b_j\}$.
Hence, the $\binom{|A_d|}{2}$ lines $\ov{b_ib_j}$ for $i, j \in A_d$ are all pairwise distinct.
\end{proof}

\begin{lem}\label{lem.graph_gamma_quad} Let $G$ be a graph and $(V, \rho)$ the corresponding graph metric.
Let $Q$ be a positive integer. Let $\{a_1, b_1\}, \ldots, \{a_Q, b_Q\}$ be $Q$ pairs of vertices such that, for any $1 \le i \neq j \le Q$, the following hold:
\begin{itemize}
\item $\ov{a_ib_i}=\ov{a_jb_j}$ and
\item $\{a_i,b_i\}$ and $\{a_j,b_j\}$ are $\gamma$-related.
\end{itemize}
Then $(V, \rho)$ has at least  $(1/2 - o(1))Q^2$ lines.
\end{lem}

\begin{proof}
The $Q$ pairs $\{a_1, b_1\}, \ldots, \{a_Q, b_Q\}$ are pairwise $\gamma$-related, so the distances $\rho(a_i, b_i)$ are all the same value, say $t$. For any $1 \le i \neq j \le Q$ we have $[a_ia_jb_i]$, $[a_jb_ib_j]$, $[b_ib_ja_i]$ and $[b_ja_ia_j]$.
Moreover, by Lemma~\ref{lem.gamma_clique}, all the $2Q$ points $a_1, \ldots, a_Q, b_1, \ldots, b_Q$ are distinct.
Set $X=\{a_i,b_i:1 \le i \le Q\}$.
\medskip

\noindent{\em Case 1:} $t = 2$.
\\
For any $1 \le i \neq j \le Q$, we have $[a_ja_ib_j]$, so $\rho(a_i,a_j)=1$.
 This implies that no three vertices among $\{a_1, \ldots, a_Q\}$ are collinear
and thus the $\binom{Q}{2}$ lines $\ov{a_ia_j}$ for $1 \leq i < j \leq Q$ are all pairwise distinct.
\medskip

\noindent{\em Case 2:} $t > 2$.
\\
For any $1 \le i < j<$, we set $x_{ij} = a_j$ if $\rho(a_j,a_i) \geq t/2$, otherwise $\rho(b_j,a_i) \geq t/2$ (because $[a_ja_ib_j]$) and we set $x_{ij} = b_j$.
So, for any $1 \le i <j \le Q$, we have $\rho(a_i,x_{ij}) \ge t/2$.

We call $\{a_i, x_{ij}\}$ a \emph{special pair}  and claim that the lines $\ov{a_ix_{ij}}$ are all distinct.
Observe that, since $\rho(a_i,x_{ij}) \ge t/2$, $I(a_i,x_{ij})\neq \emptyset$ and since $[a_ix_{ij}b_i]$, $O(a_ix_{ij})\neq \emptyset$.
Hence, if two special pairs define the same line, they must be $\alpha$-related.
Now,  let $\{a,b\}$ and $\{c,d\}$ be two special pairs and let us prove that  $\ov{ab}\neq \ov{cd}$.
For any $1 \le i \le Q$, we say that $a_i$ and $b_i$ are \emph{friends} and we let  $a'$, $b'$, $c'$ and $d'$ be respectively the friends of $a$, $b$, $c$ and $d$.
We now distinguish between the two following cases.
\medskip

\noindent{\em Case 2.1:} $\{a, b\} \cap \{c, d\} \neq \emptyset$.
\\
Assume without loss of generality that $a = d$ (and thus $a$, $b$ and $c$ are pairwise distinct).
Since $\{a,b\}$ and $\{a,c\}$ are special pairs,  $\rho(a, b)$ and $\rho(a, c)$ are both at least $t/2$, and since  $[bcb']$, $\rho(b, c) < t$.
Hence $[bac]$ does not hold.
So we may assume with out loss of generality that $[abc]$.
Then, by Lemma \ref{lem.gamma_no_mid}, $\{a,b',c\}$ is not collinear and thus $b' \notin \ov{ac}$.
But since $[bab']$, $b' \in \ov{ab}$ and thus $\ov{ab} \neq \ov{ac}$.
\medskip

\noindent{\em Case 2.2:} $a$, $b$, $c$ and $d$ are all pairwise distinct.
\\
Assume first that $\{a, b\} \cap I(c, d) = \{c, d\} \cap I(a, b) = \emptyset$.
Then, since $a,b,c,d$ are geodesic (because $\{a,b\}$ and $\{c,d\}$ are $\alpha$-related, we may assume with out loss of generality that $[abcd]$.
Hence $\rho(a,d) >\rho(a,b) + \rho(c,d)= t$, a contradiction (indeed, the maximum distance between two points of $X$ is $t$).

So we may assume with out loss of generality that $[cad]$.
Then $a \not\in \{c', d'\}$ because $[cdc']$ and $[dcd']$; then by Lemma \ref{lem.gamma_no_mid}, $a' \not\in \ov{cd} = L$.
\end{proof}

\begin{thm}\label{thm.d_graph} Any metric space induced by a graph on $n$ vertices with diameter no more than $D$ has at least
$(2^{-7/3} - o(1))(n / D)^{4/3}$ lines when $n / D \rightarrow \infty$.
\end{thm}

\begin{proof} Suppose the graph is $G = (V, E)$ with $|V| = n$. There exists $1 \leq d \leq D$ such that $(1/2D - o(1)) n^2$
pairs of vertices have the same distance $d$. Let $c = 2^{-7/3}$.

{\em Case 1.} $d=1$. There is a vertex $v$ with degree $\deg(v) \geq (1 - o(1)) n / D$. The neighbourhood $N(v)$ induces a metric subspace in which all distances are contained in $\{0, 1, 2\}$. By Theorem \ref{thm.2metric}, the pairs of vertices contained in $N(v)$ generate at least $(c- o(1))(n / D)^{4/3}$ distinct lines.

{\em Case 2.} $d>1$.  If, for every line $L$, the generator graph $H_d(L)$ has less than $Q = 2^{4/3} n^{2/3} D^{1/3}$ edges, then there are at least
$(1/2D - o(1)) n^2 / Q \geq (c-o(1)) (n/D)^{4/3}$ lines. Otherwise there is a line $L$ where $H_d(L)$ has at least $Q$ edges.

{\em Case 2.1.} At least $Q/2$ edges of $H_d(L)$ are $\gamma$-pairs. By Lemma \ref{lem.gamma_clique}, any two of these pairs are $\gamma$-related, and by Lemma \ref{lem.graph_gamma_quad}, there are at least $(1/2 -  o(1))(Q/2)^2 > (c - o(1))(n/D)^{4/3}$ lines.

{\em Case 2.2.} At least $Q/2$ edges of $H_d(L)$ are not $\gamma$-pairs. Since $d > 1$ and the metric space is a graph metric, they are also not $\beta$-pairs ($\{x, y\}$ is in a $\beta$-relation implies $I(x, y) = \emptyset$, which, in a graph metric, means $\rho(x, y) = 1$).
So any two of them are $\alpha$-related. By Lemma \ref{lem.graph_alpha_quad}, there are $(1/2 - o(1))(Q/2D)^2 \geq (c-o(1))(n/D)^{4/3}$ lines.
\end{proof}

Note that we can slightly improve the constant in the above lemma, but that is not our main goal.
Theorem~\ref{thm.d_graph} immediately implies:

\begin{thm}\label{thm.d_graphs} For any natural number $D$, any metric induced by an $n$-vertex graph with diameter $D$ has $\Omega(n^{4/3})$ lines.
\end{thm}

It also implies the following, which improves the $\Omega(\sqrt{n})$ lower bound in Theorem \ref{thm.main} for the graph metrics.

\begin{thm}\label{thm.graphs} Any graph metric on $n$ points either has at least $(1/2 - o(1))n^{4/7}$ lines or has a universal line.
\end{thm}

\begin{proof} If the diameter of the graph is at least $0.5 n^{4/7}$, then we are done by Lemma \ref{lem.long_path}. Otherwise, we apply Theorem~\ref{thm.d_graph}.
\end{proof}

Theorem~\ref{thm.d_graphs} and some other results that we detail now suggest the conjecture that all graph metrics either have universal line or define at least $\Omega(n^{4/3})$ lines.

As we said in the introduction, it was proved in Chv\'{a}tal \cite{Chvatal} that any 2-metric space with $n$ points has either a universal line or at least $n$ lines, whereas in terms of the asymptotic growth rate, Chiniforooshan and Chv\' atal \cite{Chini_Chvatal} proved that such a metric space has $\Omega(n^{4/3})$ lines. (See Theorem \ref{thm.2metric} and our extension Theorem \ref{thm.3metric}). They also gave an example showing that the bound is tight.
Their construction is actually a family of metric spaces induced by graphs:

\begin{eg} (\cite{Chini_Chvatal}) Consider the metric space induced by a complete $k$-partite graph on $n$ vertices where
$k = \lfloor n^{2/3} \rfloor$ and the vertex set is divided into the $k$ parts as evenly as possible.
There is no universal line and the number of lines is $\Theta(n^{4/3})$.
\end{eg}

Here we give another example.

\begin{eg}
For each integer $s>1$, construct a path of length $s^3$
\[P = (v_0, v_1, ..., v_{s^3}).\]
Then add $s^2$ paths disjoint from $P$ and each other, except at their end points, as follows.
For each $0 \leq j < s^2$, add a path of length $s+1$ between $v_{js}$ and $v_{js+s}$.
Let $n = s^3 + 1 + s^2(s-1)$ be the number of vertices.
There is no universal line and the number of lines is $\Theta(n^{4/3})$.
\end{eg}

In recent work, Chen, Huzhang, Miao, and Yang \cite{CHMY} studied, among other things, graph metrics where no line is a proper subset of another. They proved:

\begin{thm}\label{thm.gd} (\cite{CHMY}) Let $g(n)$ be the least number of lines over metric spaces
induced by a connected graph with $n$ vertices, and where there is no universal line and no
line is a proper subset of another. Then $g(n) \in \Omega(n^{4/3})$ and $g(n) \in O(n^{4/3} \ln^{2/3} n)$.
\end{thm}

No family of finite metric spaces induced by graphs (without a universal line) that we know of has $o(n^{4/3})$ lines. The following conjecture was first proposed by V. Chv\'{a}tal.

\begin{conj} (V. Chv\'{a}tal, personal communication) In any metric space induced by a connected $n$-vertex graph, either there is a universal line or there are $\Omega(n^{4/3})$ lines.
\end{conj}

\section*{Acknowledgement}

We thank Va\v{s}ek Chv\'atal for bringing us together and for much help he provided us. This research was undertaken, in part, during a visit of Xiaomin Chen to Concordia University, thanks to funding from the Canada Research Chairs program and from the Natural Sciences and Engineering Research Council of Canada.
We thank Va\v{s}ek, Laurent Beaudou, Ehsan Chiniforooshan, and Peihan Miao for many helpful discussions about the subject.

\end{document}